\documentclass[12pt]{amsart}
\usepackage{amsmath, amssymb, enumitem, float, frcursive, mathrsfs}
\usepackage[headings]{fullpage}

\numberwithin{equation}{section}

\theoremstyle{plain}

\newtheorem{theorem}[subsection]{Theorem}
\newtheorem{lemma}[subsection]{Lemma}
\newtheorem{proposition}[subsection]{{Proposition}}

\theoremstyle{definition}

\theoremstyle{remark}
\newtheorem{remark}[subsection]{{Remark}}

\newtheoremstyle{RepTheorem} 
	{\topsep}{\topsep}
	{\itshape}
	{}
	{\bfseries}
	{.}
	{ }
	{\thmname{#1}\thmnote{ \bfseries #3}}
\theoremstyle{RepTheorem}

\def\Z {{\mathbb Z}}
\def\F {{\mathbb F}}
\def\Q {{\mathbb Q}}

\def\P {{\mathbb P}}

\def\mcA {{\mathcal A}}

\def\mcC {{\mathcal C}}
\def\mcD {{\mathcal D}}

\def\mcF {{\mathcal F}}

\def \a {{\mathfrak a}}

\def \d {{\mathfrak d}}

\def \m {{\mathfrak m}}

\def \mfD {{\mathfrak D}}

\newcommand{\Div}{\textnormal{div}}
\newcommand{\DIV}{\textnormal{Div}}

\newcommand{\Ht}{\textnormal{ht}}

\newcommand{\Pic}{\textnormal{Pic}}

\newcommand{\Val}{\textnormal{Val}}

\begin{document}

\title{Points of Bounded Height on Projective Spaces Over Global Function Fields via Geometry of Numbers}

\author{Tristan Phillips}

\begin{abstract}
We give a new proof of a result of DiPippo and Wan for counting points of bounded height on projective spaces over global function fields. The new proof adapts the geometry of numbers arguments used by Schanuel in the number field case.
\end{abstract}


\maketitle

\section{Introduction}

Studying rational points on projective varieties is a central problem in number theory. One often studies how many rational points are on a given variety. When the variety has infinitely many rational points, one is led to considers the more refined problem of giving an asymptotic for the number of points of bounded height.

In 1979 Schanuel proved a beautiful asymptotic formula for the number of rational points of bounded Weil height on projective space $\P^n$ over a number field $K$ \cite{Sch79}. 

\begin{theorem}[Schanuel] \label{thm:Schanuel}
Let $K$ be a number field of degree $d$ over $\Q$, with class number $h$, regulator $R$, discriminant $\Delta(K)$, $r_1$ real places, $r_2$ complex places, Dedekind zeta function $\zeta_K(s)$, and which contains $\omega$ roots of unity. 
Then the number of $K$-rational points on $\P^n$ of height at most $B$ is
\[
\frac{h R (2^{r_1}(2\pi)^{r_2})^{n+1}(n+1)^{r_1+r_2-1}}{\zeta_K(n+1)|\Delta(K)|^{(n+1)/2}\omega } B^{n+1}+\begin{cases}
O(B\log(B)) & \text{ if } \P^n(K)=\P^1(\Q)\\
O(B^{n+1-1/d}) & \text{otherwise.}
\end{cases}
\]
\end{theorem}
     
 Schanuel's Theorem was independently extended to global function fields by DiPippo \cite{Dip90} and Wan \cite{Wan91}. Both of their proofs are essentially the same, studying the relevant height zeta functions and exploiting facts about zeta functions of varieties over finite fields. In this article we will give a new proof of the function field case which uses techniques from the geometry of numbers over function fields, and which more closely resembles Schanuel's proof in the number field case (albeit with a weaker error term than that obtained by DiPippo and Wan; see Theorem \ref{thm:WProjFin}).
 
 Before stating the result over global function fields, we introduce some notation. Let $\mcC$ be an absolutely irreducible smooth projective curve of genus $g$ over $\F_q$, let $K$ denote the function field of $\mcC$, and let $h_K$ denote the class number of $K$.

Let $\DIV(\mcC)$ denote the set of divisors on $X$ and let 
\[
\DIV^+(\mcC):=\left\{\sum_{P} n_P P\in\Div(\mcC):n_P\geq 0 \text{ for all } P\right\}
\]
denote the set of effective divisors on $X$.
Define by
\[
 Z(\mcC,t):=\exp\left(\sum_{d=1}^\infty \# \mcC(\F_{q^d}) \frac{t^d}{d}\right)=\sum_{D\in \DIV^+(\mcC)} t^{\deg(D)}
\]
the classical zeta function of $\mcC$, and let $\zeta_\mcC(s)=Z(\mcC,q^{-s})$ denote the Dedekind zeta function of $\mcC$.

For $f\in K$ let $\Div(f)$ denote the divisor of $f$ and let $\inf_{i}(x_i)$ denote the greatest divisor $D$ of $X$ such that $D\leq \Div(x_i)$ for all $i$.
 We define the \emph{(logarithmic) height} of a $K$-rational point of projective space $x=[x_0:x_1:\cdots:x_n]\in \P^n(K)$ by 
\[
\Ht(x):=-\deg(\inf_i(x_i)).
\] 
Set
\[
A_r(\P^n):=\{x\in \P^n(K) : \Ht(x)=r\},
\]
the number of $K$-rational points on $\P^n$ of height  $r$.

\begin{theorem}[DiPippo-Wan]\label{thm:DiPippo-Wan}
With notation as above, for any $\varepsilon>0$ there exist polynomials $b_i(r)$ in $r$, whose coefficients are algebraic numbers, and algebraic integers $\alpha_i$, with $|\alpha_i|=\sqrt{q}$, such that
\begin{align*}
A_r(\P^n)&=\frac{h_K q^{(n+1)(1-g)}}{\zeta_\mcC(n+1)(q-1)}q^{(n+1)r}+\sum_{i=1}^{2g} b_i(r)\alpha_i^r\\
&=\frac{h_K q^{(n+1)(1-g)}}{\zeta_\mcC(n+1)(q-1)}q^{(n+1)r}+O(q^{r/2+\varepsilon}).
\end{align*}
\end{theorem}

\section{Preliminaries}

Let $\mcC$ be a smooth, projective, geometrically connected curve over the finite field $\F_q$, let $K=\F_q(\mcC)$ denote the function field of $\mcC$, let $\infty$ be a fixed closed point of $\mcC$ of degree $d_\infty$ over $\F_q$, and let $\mcA\subset K$ be the ring of functions regular outside of $\infty$. 

Let $\Val(K)$ denote the set of discrete valuations of $K$; these are in bijection with the set of closed points on $\mcC$ (i.e. \emph{places}). Let $v_\infty$ be the valuation corresponding to the point $\infty$ and set $\Val_0(K):=\Val(K)-\{v_\infty\}$. For each place $v\in \Val(K)$ let $K_v$ denote the completion of $K$ with respect to $v$ and define an absolute value on $K_v$ by
\[
|x|_v:=q^{-v(x) \deg(v)}.
\]
Let $m_v$ denote the $v$-adic measure on $K_v$.
Let $\pi_v$ be a uniformizer at $v$, so that $|\pi_v|_v=q^{-\deg(v)}$.
Setting 
\[
\Ht_\infty(x):=-d_\infty \min_i\{v_\infty(x_i)\}
\]
and
\[
\Ht_0(x):=-\sum_{v\in \Val_0(K)} \deg(v) \min_i\{v(x_i)\}
\]
we may decompose the height function from the introduction as 
\[
\Ht(x)=\Ht_0(x)+ \Ht_\infty(x).
\]

Define the \emph{(finite) scaling divisor of $x$} by
\[
\mfD_0(x):=\sum_{v\in \Val_0(K)} \min_i\{v(x_i)\} P_v.
\]
Then $\Ht_0(x)=-\deg(\mfD_0(x))$.

For any $x=(x_1,\dots,x_n)\in K_\infty^n$ and $t\in K_\infty$ set $t\ast x:=(tx_1,\dots, tx_n)$.

\section{Geometry of numbers over global function fields}

The following result gives a function field analog of the Principal of Lipschitz, and is a special case of a result of Bhargava, Shankar, and Wang \cite[Proposition 33]{BSW15}.

\begin{proposition}\label{prop:FunctionFieldLipschitz}
Let $R$ be an open compact subset of $K_\infty^n$. Let $\Lambda$ be a rank $n$ lattice in $F_\infty^n$ and let $m_\Lambda$ be the constant multiple of the measure $m_\infty$ for which $m_\Lambda(K_\infty^n/\Lambda)=1$. For $t\in K_\infty$ set
\[
t\ast R = \{t\ast x : x\in R\}.
\]
Then
\[
\#\{(t\ast R) \cap \Lambda\}
= m_\Lambda (R) |t|_\infty^{n}  + O\left(|t|_\infty^{n-1}\right)=\frac{m_{\infty}(R)}{\m_\infty(K_\infty^n/\Lambda)}|t|_\infty^{n}  + O\left(|t|_\infty^{n-1}\right) ,
\]
where the implied constant depends only on $K$, $\infty$, $n$, $R$, and $\Lambda$.
\end{proposition}

As an important example, note that $\mcA^n$ may naturally be viewed as a lattice in $K_\infty^n$. Following Weil \cite[2.1.3 (b)]{Wei82}, we will compute the covolume $m_\infty(K_\infty^n/\mcA^n)$ of the lattice $\mcA^n$. First note that $m_\infty(K_\infty^n/\mcA^n)=m_\infty(K_\infty/\mcA)^n$. Since $\pi_\infty$ has a zero at $\infty$ it must have a pole somewhere in $\mcC-\infty$, and thus $\mcA\cap (\pi_\infty)=\{0\}$. Thus $K_\infty/\mcA$ has a set of representatives consisting of cosets of $(\pi_\infty)$; more specifically,
\[
K_\infty/\mcA=\bigsqcup_{F\in \{f\in K\ :\ \Div(f)\geq \infty\}} F \cdot (\pi_\infty).
\]
 By Riemann-Roch, $\#\{f\in K\ :\ \Div(f)\geq \infty\}=q^{d_\infty+g-1}$. As $m_\infty((\pi_\infty))= q^{-d_\infty}$, it follows that
 \[
 m_\infty(K_\infty^n/\mcA^n)=q^{n(g-1)}.
 \]
For $\a\subseteq \mcA$ an integral ideal of $\mcA$, let $N(\a):=\#(\mcA/\a)$ denote its norm.
Then the lattice $\Lambda_\a:=\a\times\cdots\times \a\subset K_\infty^n$ has covolume
\[
m_\infty(K_\infty^n/\Lambda_\a)=N(\a)^n q^{n(g-1)}.
\]
 By Proposition \ref{prop:FunctionFieldLipschitz} we have the following proposition:

\begin{proposition}\label{prop:FunctionFieldIdealLipschitz}
Let $R$ be an open compact subset of $K_\infty^n$. Let $\a\subset \mcA$ be an integral ideal and let $\Lambda_\a$ denote the rank $n$ lattice $\a\times\cdots\times \a$ in $F_\infty^n$.
Then
\[
\#\{(t\ast R) \cap \Lambda_\a\}
=\frac{m_{\infty}(R)}{N(\a)^n q^{n(g-1)}}|t|_\infty^{n}  + O\left(|t|_\infty^{n-1}\right) ,
\]
where the implied constant depends only on $K$, $\infty$, $n$, $R$, and $\a$.
\end{proposition}

\section{Counting points on projective spaces}\label{sec:PROJ}

\begin{theorem}\label{thm:WProjFin}
The number of $K$-rational points of bounded height on projective $n$ space over $K$ is
\begin{align*}
\#\{x\in \P^n(K): \Ht(x) = r\}
= \kappa q^{(n+1)r}
 + \begin{cases}
O\left(q^r\log(q^r)\right) &  \text{ if } n=d_\infty=1\\
O\left(q^{r(n+1-1/d_\infty)}\right) & \text{ else, }
\end{cases}
\end{align*}
where the leading coefficient is
\[
\kappa=\frac{h_K q^{(n+1)(1-g)}}{\zeta_K(n+1)(q-1)}.
\]
\end{theorem}

\begin{proof}

Let $h=h_\mcA$ denote the size of the divisor class group $\Pic(\mcA)$ of $\mcA$, and let $D_1,\dots,D_h$ be a set of effective divisor representatives of the divisor class group $\Pic(\mcA)$. Then we have the following partition of $\P^n(K)$ into points whose scaling divisors are in the same ideal class,
\[
\P^n(K)=\bigsqcup_{i=1}^h \{x\in \P^n(K): [\mfD_0(x)]=[D_i]\}.
\] 
For each $D\in \{D_1,\dots,D_h\}$ define a counting function
\[
M(D,r):=\#\{x\in \P^n(K): \Ht(x)= r,\ [\mfD_0(x)]=[D]\}.
\]

Consider the action of the unit group $\mcA^\times=\F_q^\times$ on $K^{n+1}-\{0\}$ by $u\ast(x_0,\dots,x_n):=(u x_0,\dots,u x_n)$, and let $(K^{n+1}-\{0\})/\mcA^\times$ denote the corresponding set of orbits. We may describe $M(D,r)$ in terms of $\mcA^\times$-orbits in the affine cone of $\P^n(K)$: 

\begin{lemma}
There is a bijection between the sets
\[
 \{x=[(x_0,\dots,x_n)]\in (K^{n+1}-\{0\})/\mcA^\times : \Ht_\infty(x) - \deg(D)=r,\ \mfD_0(x)=D\}
 \]
and
\[
\{x=[x_0:\cdots:x_n]\in \P^n(K): \Ht(x)=r,\ [\mfD_0(x)]=[D]\},
\]
 given by
 \[
 [(x_0,\dots,x_n)] \mapsto [x_0:\dots:x_n].
 \]
 \end{lemma}
 
 \begin{proof}
 First note that the map is well defined, for if
$x=[(x_0,\dots,x_n)]$ and $y=[(y_0,\dots,y_n)]$ are points of $(K^{n-1}-\{0\})/\mcA^\ast$ with the following properties
\begin{itemize}
\item $x=y$,
\item $\Ht_\infty(x) - \deg(D)=\Ht_\infty(y) -\deg(D)=r$, and
\item $\mfD_0(x)=\mfD_0(y)=D$,
\end{itemize}
 then
 \begin{itemize}
 \item $[x_0:\cdots: x_n]=[y_0:\cdots:y_n]$ (since there exists a $u\in \mcA^\ast$ such that $(x_0,\dots,x_n)=(uy_0,\dots,uy_n)$),
 \item $\Ht([x_0:\dots:x_n])=\Ht([y_0:\cdots:y_n])=r$, and 
 \item $[\mfD_0(x)]=[\mfD_0(y)]=[D]$.
 \end{itemize}

 To show that the map is injective, let $x=[(x_0,\dots,x_n)]$ and $y=[(y_0,\dots,y_n)]$ be elements of $(K^{n+1}-\{0\})/\mcA^\times$ satisfying $\Ht_\infty(x)=r-\deg(D)$ and $\mfD_0(x)=D$ and such that $[x_0:\cdots:x_n]=[y_0:\cdots:y_n]$. Then there exists a $\lambda\in K^\times$ such that $(\lambda x_0,\dots, \lambda x_n)=(y_0,\dots,y_n)$. But then
 \[
 \Ht_\infty(x)=\Ht_\infty(y)=\Ht_\infty(\lambda \ast x)=-d_\infty v_\infty(\lambda)+ \Ht_\infty(x).
 \]
 Thus we must have $v_\infty(\lambda)=0$, so $\lambda\in \mcA^\times$. This shows that $x=y$, so the map is injective.
 
 To show that the map is surjective, let $x=[x_0:\dots:x_n]\in \P^n(K)$ be such that $\Ht(x)=r$ and $[\mfD_0(x)]=[D]$. Since $[\mfD_0(x)]=[D]$, the fractional ideals $\mfD_0(x)$ and $D$ differ by a principal ideal, which we will denote by $(\lambda)$ with $\lambda\in \mcA$. Then $\mfD_0(\lambda\ast x)=D$ and the element $[(\lambda x_0,\dots, \lambda x_n)]$ of $(K^{n+1}-\{0\})/\mcA^\times$ satisfies $\mcD_0((\lambda x_0,\dots, \lambda x_n))=D$ and $\Ht_\infty((\lambda x_0,\dots, \lambda x_n))+\deg(D)=r$, and is in the pre-image of $[x_0:\cdots:x_n]$. Therefore the map is surjective. 
 \end{proof}

From this lemma it follows that
\[
M(D,r)=\#\{x\in (K^{n+1}-\{0\})/\mcA^\times :\Ht_{\infty}(x) - \deg(D)=r,\ \mfD_0(x)=D\}.
\]
Our general strategy will be to first find an asymptotic for the counting function
\[
M'(D,r):=\#\{x\in (K^{n+1}-\{0\})/\mcA^\times : \Ht_{\infty}(x) - \deg(D) =r,\ D\leq\mfD_0(x) \},
\]
and then use M\"obius inversion to obtain an asymptotic formula for $M(D,r)$.

Note that each $\mcA^\times$-orbit of an element of $(K-\{0\})^{n+1}$ contains exactly 
\[
\# \mcA^\times=\# \F_q^\times=(q-1)
\]
 elements. Viewing $\mcA^{n+1}$ as a lattice of full rank in $K_\infty^{n+1}$, we have that
\[
M'(D,r)=\frac{1}{q-1}\#\{x\in (K^{n+1}-\{0\}) : \Ht_{\infty}(x)=r + \deg(D),\ D\leq \mfD_0(x)\}.
\]
As $D$ is an effective divisor, if $D\leq \mfD_0(x)$, then $\mfD_0(x)$ is also an effective divisor. Therefore we must have that $x\in \mcA^{n+1}-\{0\}$, and so
\[
M'(D,r)=\frac{1}{q-1}\#\{x\in (\mcA^{n+1}-\{0\}) : \Ht_{\infty}(x)=r + \deg(D),\ D\leq \mfD_0(x)\}.
\]


Define the following open bounded subsets of $K_{\infty}^{n+1}-\{0\}$:
\[
\mcF(r):=\{x\in  K_\infty^{n+1}-\{0\} : \Ht_{\infty}(x)= r\}.
\]
Note that these sets are $\mcA^\times$-stable, in the sense that if $u\in \mcA^{\times}$ and $x\in \mcF(r)$ then $u\ast x\in \mcF(r)$; this can be seen by the following computation: 
\[
\Ht_{\infty}(u\ast x)
= -d_\infty \min_i\{v_\infty(u x_{i})\}
= -d_\infty \min_i\{v_\infty(u)+v_\infty(x_{i})\} 
= -d_\infty\min_i\{0+v_\infty(x_{i})\} 
= \Ht_{\infty}(x).
\]
 Similarly, for any $t\in K_\infty$, we have that
\[
\Ht_{\infty}(t\ast x)
= -d_\infty\min_i\{v_\infty(t x_{i})\}
= -d_\infty\min_i\{v_\infty(t)+ v_\infty(x_{i})\}
=\Ht_{\infty}(x)-d_\infty v_\infty(t).
\]
From this it follows that, for all $r\in d_\infty \Z$,
\[
\mcF(r)=\pi_\infty^{-r/d_\infty}\ast \mcF(0).
\]
As the image of $v_\infty: K_\infty\to \Z$ is $d_\infty \Z$, for $r\not\in d_\infty\Z$ we have $\mcF(r)=\emptyset$.

Let $\d\subset \mcA$ denote the ideal corresponding to the effective divisor $D$, and define the lattice 
\[
\Lambda_D:=\d\times \d \times\cdots\times \d\subset K_\infty^{n+1}.
\]
 Then
\[
M'(D,r)=\frac{\#\{(\Lambda_D-\{0\})\cap \mcF(r+\deg(D))\}}{q-1}.
\]
Applying Proposition \ref{prop:FunctionFieldIdealLipschitz} with $R=\mcF(0)$ and $\a=\d$, we obtain the following asymptotic when $r\equiv -\deg(D) \pmod{d_\infty}$,
\begin{align}\label{eq:I<=c asymptotic}
M'(D,r)&=\frac{m_\infty(\mcF(0))}{N(\d)^{n+1}q^{(n+1)(g-1)}(q-1)} (q^r N(\d))^{n+1}+ O\left(q^{r(n+1-1/d_\infty)}\right)\\
&=\frac{m_\infty(\mcF(0))}{q^{(n+1)(g-1)}(q-1)} q^{r(n+1)}+ O\left(q^{r(n+1-1/d_\infty)}\right).
\end{align}
Note that $m_\infty(\mcF(0))=1-q^{-d_\infty (n+1)}$. When $r \not\equiv \deg(D)\pmod{d_\infty}$, then $M'(D,r)=0$.

For $x\in \Lambda_D$, note that $\mfD_0(x)=D+D'$ for some effective divisor $D'$. Then
\[
M'(D,r)=\sum_{D'\in \DIV^+(\mcC)} M(D+D', r-\deg(D'))= \sum_{\substack{D'\in \DIV^+(\mcC)\\ \deg(D') \leq r}} M(D+D', r-\deg(D')).
\]

We now apply M\"obius inversion and use our asymptotic (\ref{eq:I<=c asymptotic}) for $M'(D,r)$. For $r\equiv -\deg(D) \pmod{d_\infty}$,
\begin{align*}
M(D,r) &= \sum_{\substack{D'\in \DIV^+(\mcC) \\ \deg(D') \leq r}} \mu(D') M'(D+D', r-\deg(D'))\\
&=\sum_{\substack{D'\in \DIV^+(\mcC) \\ \deg(D') \leq r}} \mu(D')\left(\frac{1-q^{d_\infty(n+1)}}{q^{(n+1)(g-1)}(q-1)}\left(q^{r-\deg(D')}\right)^{n+1}+ O\left(\left(q^{r-\deg(D')}\right)^{n+1-1/d_\infty}\right)\right)\\
&=\frac{1-q^{d_\infty(n+1)}}{q^{(n+1)(g-1)}(q-1)} q^{r(n+1)}
\left(\sum_{D'\in \DIV^+(\mcC)} \mu(D')q^{-(n+1)\deg(D')}
 - \sum_{\substack{D'\in \DIV^+(\mcC) \\ \deg(D') \leq r}}\mu(D')q^{-(n+1)\deg(D')}\right)\\ 
&\hspace{20mm} +O\left(q^{r(n+1-1/d_\infty)} \sum_{\substack{D'\in \DIV^+(\mcC) \\ \deg(D') \leq r}}q^{-(n+1-1/d_\infty)\deg(D')}\right)\\
&=\frac{1-q^{d_\infty(n+1)}}{q^{(n+1)(g-1)}(q-1)} q^{r(n+1)}\left(\frac{1}{\zeta_\mcA(n+1)}
 - O\left(q^{-rn}\right)\right) 
+ \begin{cases}
O\left(q^{r}\log(q^r)\right) & \text{ if } n=d_\infty=1\\
O\left(q^{r(n+1-1/d_\infty)}\right) & \text{ else, }
\end{cases}\\
&= \frac{1}{\zeta_K(n+1)q^{(n+1)(g-1)}(q-1)} q^{r(n+1)}
 + \begin{cases}
O\left(q^{r}\log(q^r)\right) & \text{ if } n=d_\infty=1\\
O\left(q^{r(n+1-1/d_\infty)}\right) & \text{ else. }
\end{cases}
\end{align*}
When $r\not\equiv \deg(D) \pmod{d_\infty}$ we have $M(D,r)=0$.

From the exact sequence 
\[
0 \rightarrow \Pic^0_\mcC(\F_q) \rightarrow \Pic(\mcA) \xrightarrow{\deg} \Z/d_\infty\Z \rightarrow 0
\]
(see e.g. \cite[Lemma 4.1.2]{Gos96}), 
we have that for any $a\in \Z/d_\infty \Z$ there are exactly $h_K=\#\Pic^0_\mcC(\F_q)$ elements of $\Pic(\mcA)$ of degree congruent to $a$ modulo $d_\infty$. Therefore, summing over the divisor class representatives $D_i$ of $\Pic(\mcA)$ gives
\begin{align*}
\#\{x\in \P^n(K): \Ht(x) = r\}
= \kappa q^{r(n+1)}
 + \begin{cases}
O\left(q^{r}\log(q^r)\right) & \text{ if } n=d_\infty=1\\
O\left(q^{r(n+1-1/d_\infty)}\right) & \text{ else, }
\end{cases}
\end{align*}
where
\[
\kappa=\frac{h_{K}}{\zeta_K(n+1)q^{(n+1)(g-1)}(q-1)}.
\]
\end{proof}

\begin{remark}
This proof can easily be modified to count points of bounded height on weighted projective spaces, giving an alternative proof of the main result in \cite{Phi24}. We have chosen only to present the proof of projective space in order to make the arguments more transparent.
\end{remark}

\bibliographystyle{alpha}
\bibliography{bibfile}

\begin{thebibliography}{BSW15}

\bibitem[BSW15]{BSW15}
Manjul Bhargava, Arul Shankar, and Xiaoheng Wang.
\newblock Geometry-of-numbers methods over global fields i: Prehomogeneous
  vector spaces, 2015.
\newblock Preprint, \texttt{arXiv:1512.03035}.

\bibitem[DiP90]{Dip90}
Stephen~Ascanio DiPippo.
\newblock {\em Spaces of rational functions on curves over finite fields}.
\newblock ProQuest LLC, Ann Arbor, MI, 1990.
\newblock Thesis (Ph.D.)--Harvard University.

\bibitem[Gos96]{Gos96}
David Goss.
\newblock {\em Basic structures of function field arithmetic}, volume~35 of
  {\em Ergebnisse der Mathematik und ihrer Grenzgebiete (3) [Results in
  Mathematics and Related Areas (3)]}.
\newblock Springer-Verlag, Berlin, 1996.

\bibitem[Phi24]{Phi24}
Tristan Phillips.
\newblock Points of bounded height on weighted projective spaces over global
  function fields.
\newblock {\em Ramanujan J.}, 65(2):477--486, 2024.

\bibitem[Sch79]{Sch79}
Stephen~Hoel Schanuel.
\newblock Heights in number fields.
\newblock {\em Bull. Soc. Math. France}, 107(4):433--449, 1979.

\bibitem[Wan92]{Wan91}
Da~Qing Wan.
\newblock Heights and zeta functions in function fields.
\newblock In {\em The arithmetic of function fields ({C}olumbus, {OH}, 1991)},
  volume~2 of {\em Ohio State Univ. Math. Res. Inst. Publ.}, pages 455--463. de
  Gruyter, Berlin, 1992.

\bibitem[Wei82]{Wei82}
Andr\'{e} Weil.
\newblock {\em Adeles and algebraic groups}, volume~23 of {\em Progress in
  Mathematics}.
\newblock Birkh\"{a}user, Boston, Mass., 1982.
\newblock With appendices by M. Demazure and Takashi Ono.

\end{thebibliography}

\end{document}